\documentclass[twoside,11pt]{article}
\usepackage{a4wide,cite,latexsym,amsfonts,amssymb,exscale,epsfig,hyperref,url}
\usepackage[centertags,sumlimits,intlimits,namelimits,reqno]{amsmath}

\usepackage{amsthm}
\theoremstyle{definition}
\newtheorem{theorem}{Theorem}[section]
\newtheorem{lemma}[theorem]{Lemma}
\newtheorem{corollary}[theorem]{Corollary}
\newtheorem{proposition}[theorem]{Proposition}
\newtheorem{definition}[theorem]{Definition}
\newtheorem{remark}[theorem]{Remark}
\newtheorem{example}[theorem]{Example}

\usepackage{bbm}
\def\C{{\mathbbm C}}
\def\Q{{\mathbbm Q}}
\def\N{{\mathbbm N}}
\def\R{{\mathbbm R}}
\def\Z{{\mathbbm Z}}

\def\ie{{\sl i.e.\/}}
\def\eg{{\sl e.g.\/}}

\let\phi=\varphi
\let\theta=\vartheta
\let\epsilon=\varepsilon

\def\Char{\mathop{\rm char}\nolimits}

\def\id{\mathop{\rm id}\nolimits}

\def\Frac{\mathop{\rm Frac}\nolimits}
\def\Aut{\mathop{\rm Aut}\nolimits}

\renewcommand{\det}{\mathop{\rm det}\nolimits}

\def\lim{\mathop{\rm lim}\limits}

\let\hat=\widehat
\let\tilde=\widetilde

\def\pprime{{\prime\prime}}
\def\ppprime{{\prime\prime\prime}}

\def\q#1{{[#1]}_q}

\numberwithin{equation}{section}

\makeatletter

\newfont{\@aidxte}{cmsy10}
\newfont{\@aidxel}{cmsy10 scaled 1095}
\newfont{\@aidxtw}{cmsy10 scaled 1200}
\newlength\@aidxtexvi
\newlength\@aidxtexvii
\newlength\@aidxelxvi
\newlength\@aidxelxvii
\newlength\@aidxtwxvi
\newlength\@aidxtwxvii
\newcommand{\alignidx}[1]{%
\@aidxtexvi=\fontdimen16\@aidxte
\@aidxtexvii=\fontdimen17\@aidxte
\@aidxelxvi=\fontdimen16\@aidxel
\@aidxelxvii=\fontdimen17\@aidxel
\@aidxtwxvi=\fontdimen16\@aidxtw
\@aidxtwxvii=\fontdimen17\@aidxtw
{\mbox{$%
\fontdimen16\@aidxte=2.9pt
\fontdimen17\@aidxte=2.9pt
\fontdimen16\@aidxel=3.1pt
\fontdimen17\@aidxel=3.1pt
\fontdimen16\@aidxtw=3.3pt
\fontdimen17\@aidxtw=3.3pt
#1$}}%
\fontdimen16\@aidxte=\@aidxtexvi
\fontdimen17\@aidxte=\@aidxtexvii
\fontdimen16\@aidxel=\@aidxelxvi
\fontdimen17\@aidxel=\@aidxelxvii
\fontdimen16\@aidxtw=\@aidxtwxvi
\fontdimen17\@aidxtw=\@aidxtwxvii}

\makeatother

\newenvironment{myenumerate}{%
\begin{enumerate}
\setlength{\partopsep}{0pt}
\setlength{\parskip}{0pt}}{\end{enumerate}}

\usepackage[ps,matrix,arrow,frame,curve]{xy}
\everyxy={0;<1mm,0mm>:}

\def\nn{\notag}

\def\emph#1{{\sl #1\/}}
\def\sym#1{{\mathcal #1}}
\def\bar#1{\overline{#1}}%

\def\op{\mathrm{op}}
\def\cop{\mathrm{cop}}

\def\Vect{\mathbf{Vect}}

\def\msc#1{\noindent{\small Mathematics Subject Classification (2010):
#1\par}}%
\def\keywords#1{\noindent {\small keywords: #1\par}}%

\begin{document}

\title{Weak Bialgebras of Fractions}
\author{Steve Bennoun\thanks{E-mail: \texttt{bennoun@math.ubc.ca}} and
Hendryk Pfeiffer\thanks{E-mail: \texttt{pfeiffer@math.ubc.ca}}}
\date{\small{Department of Mathematics, The University of British Columbia,\\
1984 Mathematics Road, Vancouver, BC, V2T 1Z2, Canada}\\[1ex]
December 23, 2012}

\maketitle

\begin{abstract}

We construct the algebra of fractions of a Weak Bialgebra relative to a
suitable denominator set of group-like elements that is \emph{almost
central}, a condition we introduce in the present article which is
sufficient in order to guarantee existence of the algebra of fractions and
to render it a Weak Bialgebra. The monoid of all group-like elements of a
coquasi-triangular Weak Bialgebra, for example, forms a suitable set of
denominators as does any monoid of central group-like elements of an
arbitrary Weak Bialgebra. We use this technique in order to construct new
Weak Bialgebras whose categories of finite-dimensional comodules relate to
$SL_2$-fusion categories in the same way as $GL(2)$ relates to $SL(2)$.

\end{abstract}

\msc{%
16S85,
16T05,
16T10,
}

\section{Introduction}

The notion of the field of fractions $\Frac(R)$ of an integral domain $R$ can
be generalized in various directions. In the case of the field of fractions,
the ring $R$ is commutative and free of zero-divisors, and the set of
denominators is the set of all regular elements $R\backslash\{0\}$. Before we
give an outline of the present manuscript, we sketch how the notion of the
field of fractions has been extended to more general rings $R$, to more
general sets of denominators, and finally to rings and algebras with
additional structure.

The notion of the field of fractions can be generalized to more general rings
$R$ as follows, see, for example~\cite{Or31,Be99}. Let $R$ be an associative
unital ring, not necessarily commutative, and $S\subseteq R$ be the set of
regular elements of $R$, \ie\ $S$ consists of all elements $s\in R$ such that
for all $r\in R$, the condition that $sr=0$ or $rs=0$, implies that $r=0$. If
$R$ satisfies the \emph{right Ore condition}
\begin{myenumerate}
\item[(O1)]
For all $a\in R$ and $s\in S$, there exist $b\in R$ and $t\in S$ such that $sb=at$,
\end{myenumerate}
then there exists a right ring of fractions $R[S^{-1}]$ that generalizes the
notion of the field of fractions. It comes with an injective ring homomorphism
$\phi\colon R\to R[S^{-1}]$ and satisfies the following universal property:
For each ring homomorphism $\psi\colon R\to T$ such that $\psi(s)$ is
invertible in $T$ for all $s\in S$, there exists a unique ring homomorphism
$\sigma\colon R[S^{-1}]\to T$ such that the following diagram commutes:
\begin{equation}
\label{eq_preuniversal}
\begin{aligned}
\xymatrix{
R\ar[rr]^{\phi}\ar[ddrr]_{\psi}&&R[S^{-1}]\ar[dd]^{\sigma}\\
\\
&&T.
}
\end{aligned}
\end{equation}
If $R$ is commutative or, more generally, if the denominators are central,
\ie\ $S\subseteq Z(R)$, the right Ore condition is automatically satisfied.

It is possible to generalize this construction even further to denominators
that are not necessarily regular, see, for example~\cite{St75}. In this case,
$\phi\colon R\to R[S^{-1}]$ will in general no longer be injective. Let $R$ be
an associative unital ring and $S\subseteq R$ be a multiplicative
submonoid. Then the right ring of fractions $R[S^{-1}]$ exists if and only if
the right Ore condition~(O1) and the following condition holds:
\begin{myenumerate}
\item[(S2)]
For all $s\in S$ and $a\in R$ with $sa=0$, there exists some $t\in S$ such that
$at=0$.
\end{myenumerate}
Again, if $R$ is commutative or, more generally, if the denominators are
central, $S\subseteq Z(R)$, the conditions~(O1) and~(S2) are satisfied
automatically.

A second direction of generalizations studies additional structure on the ring
$R$. Hayashi~\cite{Ha92a} has shown that the conditions~(O1) and~(S2) hold in
coquasi-triangular bialgebras $R$. This can be seen as a weakening of the
assumptions in the case in which it is the commutativity of $R$ that implies
the right Ore condition. In order for the algebra of fractions $R[S^{-1}]$ to
admit a bialgebra structure, the denominators $S$ are required to be
group-like.

In the present article, we extend this direction of generalizations further
from bialgebras to Weak Bialgebras and prove
\begin{theorem}
Let $R$ be a Weak Bialgebra and $S\subseteq R$ be a multiplicative monoid of
group-like elements that is \emph{almost central} (a condition that we
introduce for this purpose in Definition~\ref{def_almostcentral}). Then the
right ring of fractions $R[S^{-1}]$ exists, forms a Weak Bialgebra, and
satisfies the universal property~\eqref{eq_preuniversal} in the category of
Weak Bialgebras.
\end{theorem}

Our condition of being \emph{almost central} is a consequence of Hayashi's
assumptions whenever $R$ is a coquasi-triangular bialgebra. It implies the
conditions~(O1) and~(S2) in this case, and also if $R$ is commutative or if
$S$ is central.

Weak Bialgebras are relevant to the study of modular categories~\cite{Pf09a}
and, more generally, fusion categories, see, for example~\cite{Os03a,Pf11}. We
use our construction of Weak Bialgebras of fractions in order to introduce
new Weak Bialgebras whose categories of finite-dimensional comodules relate to
$SL_2$-fusion categories in the same way as $GL(2)$ relates to $SL(2)$.

The present article is organized as follows. In Section~\ref{sect_prelim}, we
review the basics of Weak Bialgebras and of localization in rings. In
Section~\ref{sect_new}, we construct Weak Bialgebras of fractions and relate
our results to some known special cases and to the weak Hopf envelope. In
Section~\ref{sect_example}, we finally present a number of examples in order
to probe the assumptions of our main theorems. We also use our results in
order to construct new Weak Bialgebras related to $SL_2$-fusion categories.

\section{Preliminaries}
\label{sect_prelim}

In this section, we fix our notation and review some background
material on Weak Bialgebras (WBAs), Weak Hopf Algebras (WHAs) and on the
classical theory of associative rings of fractions.

\subsection{Weak Bialgebras and Weak Hopf Algebras}

For the basics of Weak Bialgebras (WBAs) and Weak Hopf Algebras (WHAs), we
refer to~\cite{BoNi99,BoSz00} and for coquasi-triangular WBAs to~\cite{Pf09a}.

\begin{definition}
A \emph{Weak Bialgebra} $(H,\mu,\eta,\Delta,\epsilon)$ over a field
$k$ is a $k$-vector space $H$ such that
\begin{myenumerate}
\item
$(H,\mu,\eta)$ forms an associative algebra with multiplication $\mu\colon
H\otimes H\to H$ and unit $\eta\colon k\to H$,
\item
$(H,\Delta,\epsilon)$ forms a coassociative coalgebra with comultiplication
$\Delta\colon H\to H\otimes H$ and counit $\epsilon\colon H\to k$,
\item
the following compatibility conditions hold:
\begin{eqnarray}
\Delta\circ\mu
&=& (\mu\otimes\mu)\circ(\id_H\otimes\sigma_{H,H}\otimes\id_H)\circ(\Delta\otimes\Delta),\\
\epsilon\circ\mu\circ(\mu\otimes\id_H)
&=& (\epsilon\otimes\epsilon)\circ(\mu\otimes\mu)\circ(\id_H\otimes\Delta\otimes\id_H)\nn\\
&=& (\epsilon\otimes\epsilon)\circ(\mu\otimes\mu)\circ(\id_H\otimes\Delta^\op\otimes\id_H),\\
(\Delta\otimes\id_H)\circ\Delta\circ\eta
&=& (\id_H\otimes\mu\otimes\id_H)\circ(\Delta\otimes\Delta)\circ(\eta\otimes\eta)\nn\\
&=& (\id_H\otimes\mu^\op\otimes\id_H)\circ(\Delta\otimes\Delta)\circ(\eta\otimes\eta).
\end{eqnarray}
\end{myenumerate}
Here $\sigma_{V,W}\colon V\otimes W\to W\otimes V$, $v\otimes w\mapsto
w\otimes v$ is the transposition of the tensor factors, and by
$\Delta^\op=\sigma_{H,H}\circ\Delta$ and $\mu^\op=\mu\circ\sigma_{H,H}$, we
denote the \emph{opposite comultiplication} and \emph{opposite
multiplication}, respectively. We tacitly identify the vector spaces
$(V\otimes W)\otimes U\cong V\otimes(W\otimes U)$ and $V\otimes k\cong V\cong
k\otimes V$, exploiting the coherence theorem for the monoidal category
of vector spaces.

A \emph{homomorphism} $\phi\colon H\to H^\prime$ of WBAs over the same field
$k$ is a $k$-linear map that is a homomorphism of unital algebras as well as a
homomorphism of counital coalgebras.
\end{definition}

In a WBA $H$, there are two important linear idempotents, the \emph{source
counital map}
\begin{equation}
\epsilon_s:=(\id_H\otimes\epsilon)\circ(\id_H\otimes\mu)\circ(\sigma_{H,H}\otimes\id_H)
\circ(\id_H\otimes\Delta)\circ(\id_H\otimes\eta)\colon H\to H
\end{equation}
and the \emph{target counital map}
\begin{equation}
\epsilon_t:=(\epsilon\otimes\id_H)\circ(\mu\otimes\id_H)\circ(\id_H\otimes\sigma_{H,H})
\circ(\Delta\otimes\id_H)\circ(\eta\otimes\id_H)\colon H\to H.
\end{equation}
Their images $H_s:=\epsilon_s(H)$ and $H_t:=\epsilon_t(H)$ are
mutually commuting unital subalgebras and are called the \emph{source
base algebra} and the \emph{target base algebra}, respectively.

\begin{definition}
A \emph{Weak Hopf Algebra} $(H,\mu,\eta,\Delta,\epsilon,S)$ is a Weak
Bialgebra $(H,\mu,\eta,\Delta,\epsilon)$ with a linear map $S\colon
H\to H$ (\emph{antipode}) that satisfies the following conditions:
\begin{eqnarray}
\mu\circ(\id_H\otimes S)\circ\Delta &=& \epsilon_t,\\
\mu\circ(S\otimes\id_H)\circ\Delta &=& \epsilon_s,\\
\mu\circ(\mu\otimes\id_H)\circ(S\otimes\id_H\otimes S)
\circ(\Delta\otimes\id_H)\circ\Delta&=&S.
\end{eqnarray}
A \emph{homomorphism} $\phi\colon H\to H^\prime$ of WHAs is a homomorphism of
WBAs. The condition $\phi\circ S=S^\prime\circ\phi$ holds automatically.
\end{definition}

For convenience, we write $1=\eta(1)$ and omit parentheses in products,
exploiting associativity. We also use Sweedler's notation and write
$\Delta(x)=x^\prime\otimes x^\pprime$ for the comultiplication of $x\in H$ as
an abbreviation of the expression $\Delta(x)=\sum_k a_k\otimes b_k$ with some
$a_k,b_k\in H$. Similarly, we write
$((\Delta\otimes\id_H)\circ\Delta)(x)=x^\prime\otimes x^\pprime\otimes
x^\ppprime$, exploiting coassociativity.

A WBA [WHA] is a bialgebra [Hopf algebra] if and only if
$\epsilon_s=\eta\circ\epsilon$, if and only if
$\epsilon_t=\eta\circ\epsilon$, if and only if $H_s\cong k$, and if
and only if $H_t\cong k$.

\begin{definition}
\label{def_coquasi}
A \emph{coquasi-triangular} WBA $(H,\mu,\eta,\Delta,\epsilon,r)$ over $k$ is a
WBA $(H,\mu,\eta,\Delta,\epsilon)$ over $k$ with a linear form $r\colon
H\otimes H\to k$ (\emph{universal $r$-form}) that satisfies the following
conditions:
\begin{myenumerate}
\item
For all $x,y\in H$,
\begin{equation}
\label{eq_coquasidef}
r(x\otimes y)=\epsilon(x^\prime y^\prime)r(x^\pprime\otimes y^\pprime)
=r(x^\prime\otimes y^\prime)\epsilon(y^\pprime x^\pprime).
\end{equation}
\item
There exists a linear form $\bar r\colon H\otimes H\to k$ such that for all
$x,y\in H$,
\begin{eqnarray}
\label{eq_coquasiinv1}
\bar r(x^\prime\otimes y^\prime)r(x^\pprime\otimes y^\pprime)&=&\epsilon(yx),\\
\label{eq_coquasiinv2}
r(x^\prime\otimes y^\prime)\bar r(x^\pprime\otimes y^\pprime)&=&\epsilon(xy).
\end{eqnarray}
\item
For all $x,y,z\in H$,
\begin{eqnarray}
x^\prime y^\prime r(x^\pprime\otimes y^\pprime)
&=&r(x^\prime\otimes y^\prime)y^\pprime x^\pprime,\\
\label{eq_univr1}
r((xy)\otimes z)&=&r(y\otimes z^\prime) r(x\otimes z^\pprime),\\
\label{eq_univr2}
r(x\otimes (yz))&=&r(x^\prime\otimes y) r(x^\pprime\otimes z).
\end{eqnarray}
\end{myenumerate}
\end{definition}

We call the linear form $\bar r$ that appears in~\eqref{eq_coquasiinv1}
and~\eqref{eq_coquasiinv2} the \emph{weak convolution inverse} of $r$. Note
that it is uniquely determined by $r$ as soon as one
imposes~\eqref{eq_coquasidef}, \eqref{eq_coquasiinv1}
and~\eqref{eq_coquasiinv2}. Also note that every commutative WBA $H$ is
coquasi-triangular with $r(x\otimes y)=\epsilon(xy)$ for all $x,y\in H$. If
$H$ is a bialgebra, then Definition~\ref{def_coquasi} reduces to the familiar
notion that is sometimes also called a \emph{cobraided bialgebra}~\cite{Ka95}
or a \emph{dual quasi-triangular bialgebra}~\cite{Ma95a}.

\subsection{Localization}

In this section, we briefly review the construction of rings of fractions,
see, for example~\cite{St75,Be99}. The rings used in the following are associative
and unital, and the term \emph{ring homomorphism} always refers to
homomorphisms of associative unital rings.

\begin{definition}
\label{def_ringfractions}
Let $R$ be an associative unital ring and $S\subseteq R$ be a multiplicative
submonoid. A \emph{right ring of fractions} $R[S^{-1}]$ is a unital
associative ring with a homomorphism $\phi\colon R\to R[S^{-1}]$ such that the
following conditions hold:
\begin{myenumerate}
\item[(F1)]
For all $s\in S$, $\phi(s)$ is invertible in $R[S^{-1}]$.
\item[(F2)]
Each $a\in R[S^{-1}]$ is of the form $a=\phi(r)\phi(s)^{-1}$ for some $r\in
R$, $s\in S$.
\item[(F3)]
For each $r\in R$, we have $\phi(r)=0$ if and only if $rs=0$ for some $s\in
S$.
\end{myenumerate}
\end{definition}

The right ring of fractions $R[S^{-1}]$ is often called the
\emph{localization} of $R$ at the monoid $S$. The set $S$ is called the
\emph{set of denominators}.

\begin{proposition}[Proposition~II.1.1 of~\cite{St75}]
\label{prop_universal}
Let $R$ and $S$ be as in Definition~\ref{def_ringfractions}. If a right ring
of fractions $R[S^{-1}]$ exists, then it satisfies the following universal
property: For each ring homomorphism $\psi\colon R\to T$ such that $\psi(s)$
is invertible in $T$ for each $s\in S$, there exists a unique ring homomorphism
$\sigma\colon R[S^{-1}]\to T$ such that the following diagram commutes:
\begin{equation}
\label{eq_universal}
\begin{aligned}
\xymatrix{
R\ar[rr]^{\phi}\ar[ddrr]_{\psi}&&R[S^{-1}]\ar[dd]^{\sigma}\\
\\
&&T.
}
\end{aligned}
\end{equation}
\end{proposition}

As a consequence, the right ring of fractions $R[S^{-1}]$ is unique up to
unique ring isomorphism as soon as it exists. Recall that an element $s\in S$
is called \emph{regular in} $R$ if $sr=0$ or $rs=0$ implies that $r=0$. The
set of all regular elements of $R$ forms a multiplicative monoid. As a
consequence of the condition~(F3), the ring homomorphism $\phi\colon R\to
R[S^{-1}]$ of Definition~\ref{def_ringfractions} is injective if and only if
each element $s\in S$ is regular in $R$.

\begin{theorem}[Proposition~II.1.4 of~\cite{St75}]
\label{thm_ringfractions}
Let $R$ and $S$ be as in Definition~\ref{def_ringfractions}. The right ring of
fractions $R[S^{-1}]$ exists if and only if the following conditions are
satisfied:
\begin{myenumerate}
\item[(S1)]
For all $s\in S$ and $a\in R$, there exist $t\in S$ and $b\in R$ such that
$sb=at$.
\item[(S2)]
For all $s\in S$ and $a\in R$ with $sa=0$, there exists a $t\in S$ such that
$at=0$.
\end{myenumerate}
In this case, $R[S^{-1}]=(R\times S)/\sim$. Here, the equivalence relation
$\sim$ is defined on $R\times S$ as follows: For $a,b\in R$ and $s,t\in S$, we
have $(a,s)\sim(b,t)$ if and only if there exist $c,d\in R$ such that $ac=bd$
and $sc=td$ with $td\in S$. We denote the equivalence class of $(a,s)\in
R\times S$ in $R[S^{-1}]$ by $\frac{a}{s}=[(a,s)]\in R[S^{-1}]$. The ring
operations are given as follows.
\begin{myenumerate}
\item
For all $a,b\in R$ and $s,t\in S$,
\begin{equation}
\frac{a}{s}+\frac{b}{t} = \frac{ac+bd}{u},
\end{equation}
where $c,d\in R$ such that $u=sc=td\in S$.
\item
The zero of $R[S^{-1}]$ is given by $\frac{0}{1}$.
\item
For all $a,b\in R$ and $s,t\in S$,
\begin{equation}
\frac{a}{s}\cdot\frac{b}{t}=\frac{ac}{tu},
\end{equation}
where $c\in R$ and $u\in S$ such that $sc=bu$.
\item
The multiplicative unit of $R[S^{-1}]$ is given by $\frac{1}{1}$.
\end{myenumerate}
The homomorphism $\phi$ finally reads
\begin{equation}
\phi\colon R\to R[S^{-1}],\quad a\mapsto \frac{a}{1}.
\end{equation}
\end{theorem}

\begin{remark}
\label{rem_ore}
\begin{myenumerate}
\item
Note that the right ring of fractions $R[S^{-1}]$ as constructed in
Theorem~\ref{thm_ringfractions} satisfies the universal
property of~\eqref{eq_universal} with the unique homomorphism
$\sigma(\frac{a}{s})=\psi(a){\psi(s)}^{-1}$ for $\frac{a}{s}\in R[S^{-1}]$.
\item
Let $R$ be an associative unital ring and $S\subseteq R$ be the
multiplicative submonoid of all regular elements of $R$. In this case, the
condition
\begin{myenumerate}
\item[(O1)]
For all $a,s\in R$ such that $s$ is regular, there exist $b,t\in R$ with $t$
regular such that $sb=at$.
\end{myenumerate}
implies~(S1) of Theorem~\ref{thm_ringfractions} while~(S2) holds trivially
with $t=1$. The condition~(O1) is known as the \emph{right Ore condition} in
reference to~\cite{Or31}, see, for example~\cite[Theorem A.5.5]{Be99} who
works with left rings of fractions though.
\item
Let $R$ be an associative unital ring and $S\subseteq Z(R)$ be any
multiplicative monoid contained in its centre. In this case, both
conditions~(S1) and~(S2) of Theorem~\ref{thm_ringfractions} hold trivially.
\item
The classical case of the field of fractions of an integral domain $R$ can
be seen as a special case of either Part~(2.) or Part~(3.) above with
$S=R\backslash\{0\}$.
\end{myenumerate}
\end{remark}

Note that if the right Ore condition holds in $R$, then for $a,s\in R$ both
regular, there exist $b,t\in R$ both regular such that $sb=at$. This allows us
to simplify the equivalence relation in Theorem~\ref{thm_ringfractions} such
as to require that $c,d\in S$ (rather than merely $c,d\in R$),
see, for example~\cite[Lemma A.5.3]{Be99}, as follows:

\begin{remark}
\label{rem_regular}
Let $R$ be an associative unital ring that satisfies the right Ore
condition~(O1) and $S\subseteq R$ be the multiplicative submonoid of all
regular elements in $R$. Then the conditions~(S1) and~(S2) of
Theorem~\ref{thm_ringfractions} are satisfied and $R[S^{-1}]=(R\times S)/\sim$
where $(a,s)\sim(b,t)$ if and only if there exist $c,d\in S$ such that $ac=bd$
and $sc=td$.
\end{remark}

\begin{example}
\label{ex_localizecentral}
Let $R$ be an associative unital ring and $S=\{1,t,t^2,\ldots\}$ be the
commutative monoid generated by a central element $t\in Z(R)$ of infinite
order\footnote{By this, we mean that $t^\ell\neq t^m$ for all
$\ell\in\{1,2,3,\ldots\}$ and for all $0\leq m<\ell$.}. In this case, the
ring of fractions $R[S^{-1}]$ can be expressed in terms of polynomials with
coefficients in $R$ as follows,
\begin{equation}
R[X]/(tX-1)\cong R[S^{-1}].
\end{equation}
This can be seen by verifying the universal property
(Proposition~\ref{prop_universal}) which also yields a characterization of the
isomorphism $\Phi\colon R[X]/(tX-1)\to R[S^{-1}]$ by the conditions that
$\Phi(r)=\frac{r}{1}$ for all $r\in R$ and $\Phi(X)=\frac{1}{t}$.
\end{example}

\section{Localization in Weak Bialgebras}
\label{sect_new}

\subsection{Group-like elements}

\begin{definition}
An element $g\in H$ of a WBA $H$ is called
\begin{myenumerate}
\item
\emph{right group-like} if $\Delta g = g1^\prime\otimes g1^\pprime$ and $\epsilon_s(g)=1$,
\item
\emph{left group-like} if $\Delta g = 1^\prime g\otimes1^\pprime g$ and $\epsilon_t(g)= 1$,
\item
\emph{group-like} if it is both right and left group-like.
\end{myenumerate}
\end{definition}

The set of group-like elements in a WBA forms a monoid. Note that we do not
require the group-like elements of a WBA to have a multiplicative inverse. If
a group-like element has one, then the inverse is group-like, too. If $H$ is a
WHA with antipode $S$, then every group-like $g\in H$ has the inverse
$g^{-1}=S(g)$.

In a bialgebra, the notions of both right and left group-like elements reduces
to the conditions that $\Delta(g)=g\otimes g$ and $\epsilon(g)=1$.

\subsection{Coquasi-triangular WBAs}

The following technical results are elementary, but it takes quite some
experience with WBAs in order to establish them from the definitions. They
will be used in various proofs below. Note that an antipode is not required in
the following.

\begin{proposition}
Let $H$ be a coquasi-triangular WBA over some field $k$ with the universal $r$-form $r\colon
H\otimes H\to k$ whose weak convolution inverse we denote by
$\overline{r}\colon H\otimes H\to k$. Then for all $x,y\in H$,
\begin{gather}
\label{eq_propr1}
r(x\otimes 1) = \epsilon(x) = r(1\otimes x),\\
r(x^\prime\otimes y^\prime)r(x^\pprime\otimes\epsilon_s(y^\pprime))=r(x\otimes y),\\
r(x^\prime\otimes\epsilon_t(y^\prime))r(x^\pprime\otimes y^\pprime)=r(x\otimes y),\\
r(\epsilon_s(x^\pprime)\otimes y^\prime)r(x^\prime\otimes y^\pprime)=r(x\otimes y),\\
r(x^\pprime\otimes y^\prime)r(\epsilon_t(x^\prime)\otimes y^\pprime)=r(x\otimes y),\\
r(x^\prime\otimes y^\prime)\epsilon_t(y^\pprime)\otimes\epsilon_s(x^\pprime)
= \epsilon_t(x^\prime)\epsilon_s(y^\prime)r(x^\pprime\otimes y^\pprime),\\
r(x\otimes\epsilon_t(y))=\epsilon(xy),\\
r(x\otimes\epsilon_s(y))=\epsilon(yx),\\
r(\epsilon_t(x)\otimes y)=\epsilon(\epsilon_t(x)y),\\
r(\epsilon_s(x)\otimes y)=\epsilon(y\epsilon_s(x)),\\
r(x\otimes y^\prime)\epsilon_s(y^\pprime)
= r(x^\prime\otimes y)\epsilon_s(\epsilon_t(x^\pprime)),\\
\epsilon_t(y^\prime)r(x\otimes y^\pprime)
= \epsilon_t(\epsilon_s(x^\prime))r(x^\pprime\otimes y).
\end{gather}
\end{proposition}

In order to obtain the corresponding identities for $\overline{r}$, we use the
fact that if $(H,r)$ is coquasi-triangular, then $(H^{\op,\cop},\overline{r})$
is coquasi-triangular as well.

\begin{proposition}
Let $H$ be a coquasi-triangular WBA over some field $k$ with the universal $r$-form $r\colon
H\otimes H\to k$ whose weak convolution inverse we denote by
$\overline{r}\colon H\otimes H\to k$. Let $g\in H$ be group-like and $x,y\in
H$. Then
\begin{gather}
\overline{r}(x^\prime\otimes g)r(x^\pprime\otimes g)=\epsilon(x),\\
r(x^\prime\otimes g)\overline{r}(x^\pprime\otimes g)=\epsilon(x),\\
r(x\otimes g^\prime)r(y\otimes g^\pprime)
=r(x^\prime\otimes g)r(y^\prime\otimes g)\epsilon(y^\pprime x^\pprime),\\
\overline{r}(x^\prime\otimes g^\prime)\overline{r}(y^\prime\otimes g^\pprime)
=\epsilon(x^\prime y^\prime)\overline{r}(x^\pprime\otimes g)\overline{r}(y^\pprime\otimes g).
\end{gather}
\end{proposition}

For a WBA $H$, we denote by $\Aut(H)$ its group of automorphisms.

\begin{definition}
\label{def_almostcentral}
Let $H$ be a WBA. A multiplicative monoid $G\subseteq H$ is called
\emph{almost central} if there exists a homomorphism of monoids $\mathcal{I}\colon
G\to\Aut(H)$, $g\mapsto\mathcal{I}_g$, such that the following conditions hold:
\begin{myenumerate}
\item(C1)
$gx=\mathcal{I}_g(x)g$ for all $g\in G$ and $x\in H$,
\item(C2)
$\mathcal{I}_g(G)\subseteq G$ for all $g\in G$.
\end{myenumerate}
\end{definition}

Our terminology \emph{almost central} is motivated by the observation that if
$G\subseteq Z(H)$, then it is almost central with $\mathcal{I}_g=\id_H$ for
all $g\in G$.

\begin{proposition}
\label{prop_almostcentral}
Let $H$ be a coquasi-triangular WBA over some field $k$ with the universal $r$-form $r\colon
H\otimes H\to k$ whose weak convolution inverse we denote by
$\overline{r}\colon H\otimes H\to k$. We define for each $g,x\in G$,
\begin{equation}
\mathcal{I}_g(x)=\overline{r}(x^\prime\otimes g)x^\pprime
r(x^\ppprime\otimes g).
\end{equation}
Let $G\subseteq H$ be a monoid of group-like elements such that
$\mathcal{I}(G)\subseteq G$. Then $G$ is almost central in $H$.
\end{proposition}

\begin{proof}
It can be shown that $\mathcal{I}_g$ for each $g\in G$ is a homomorphism of
WBAs whose inverse is given by
\begin{equation}
\mathcal{I}_g^{-1}(x)=r(x^\prime\otimes g)x^\pprime\overline{r}(x^\ppprime\otimes g)
\end{equation}
for all $x\in H$. We also find that $\mathcal{I}_1=\id_H$ and
$\mathcal{I}_h\circ\mathcal{I}_g=\mathcal{I}_{hg}$ for all $h,g\in G$.
\end{proof}

\begin{remark}
If $H$ in Proposition~\ref{prop_almostcentral} is a coquasi-triangular
bialgebra and $g,h\in H$ are group-like, then
$\mathcal{I}_g(h)=h$~\cite{Ha92a}. In this case, condition~(C2) of
Definition~\ref{def_almostcentral} is always satisfied, and (C1) implies for
all $g,h\in G$ that
$gh=\mathcal{I}_g(h)g=\mathcal{I}_g(h)\mathcal{I}_g(g)=\mathcal{I}_g(hg)=hg$,
and so any monoid of group-likes that is almost central with respect to this
$\mathcal{I}$, is in fact commutative.

In a coquasi-triangular WBA this is not necessarily true as can be seen using
the results of~\cite{Ni02} as follows: Let $H$ be a WHA. The source base
algebra $H_s$ forms the monoidal unit of the category of finite-dimensional
right $H$-comodules using its right-regular comodule structure
$\beta_{H_s}\colon H_s\to H_s\otimes H$, $x\mapsto x^\prime\otimes x^\pprime$
(note that $x\in H_s$ implies that $\beta_{H_s}(x)\in H_s\otimes H$).

Each group-like element $\gamma\in H$ gives rise to a right $H$-comodule
structure on the vector space $H_s$ which we call $H_s^{(\gamma)}$,
\begin{equation}
\beta_{H_s^{(\gamma)}}\colon H_s\to H_s\otimes H,\quad
x\mapsto x^\prime\otimes(gx^\pprime).
\end{equation}
If $\gamma_1,\gamma_2\in H$ are group-like, then $H_s^{(\gamma_1)}\otimes
H_s^{(\gamma_2)}\cong H_s^{(\gamma_1\gamma_2)}$ are isomorphic as right
$H$-comodules. Furthermore, $H_s^{(\gamma_1)}\cong H_s^{(\gamma_2)}$ if and
only if $\gamma_1=\gamma_2\delta$ where $\delta$ is a so-called \emph{trivial
group-like element}, \ie\ $\delta$ is group-like and is contained in
$H_\mathrm{min}\subseteq H$, the \emph{minimal WBA} of $H$, \ie\ the smallest
sub WBA of $H$ that contains the unit of $H$.

If $H$ is a coquasi-triangular WBA, then its category of finite-dimensional
right comodules is braided monoidal. Therefore, for any group-like elements
$g,h\in H$, the comodules $H_s^{(gh)}\cong H_s^{(g)}\otimes H_s^{(h)}\cong
H_s^{(h)}\otimes H_t^{(g)}\cong H_s^{(hg)}$ are isomorphic, and so
$gh=hg\delta$ with some trivial group-like element $\delta$. Therefore, if
$H_\mathrm{min}$ contains group-like elements other than $1\in H$, and
$G\subseteq H$ is the monoid of all group-like elements, in general
$\mathcal{I}_g(h)\neq h$.
\end{remark}

\subsection{Localization}

\begin{theorem}
\label{thm_localizewba1}
Let $H$ be a WBA with an almost central multiplicative monoid $G\subseteq H$
via $\mathcal{I}\colon G\to\Aut(H)$. Then the conditions~(S1) and~(S2) of
Theorem~\ref{thm_ringfractions} are satisfied, and the right ring of fractions
$H[G^{-1}]$ exists with
\begin{gather}
\frac{x}{g}+\frac{y}{h} = \frac{x\mathcal{I}_g^{-1}(h)+yg}{hg}
=\frac{x\mathcal{I}_h^{-1}(g)+yh}{gh},\\
\frac{x}{g}\cdot\frac{y}{h} = \frac{x\mathcal{I}_g^{-1}(y)}{hg}.
\end{gather}
It forms a $k$-algebra with $\lambda\frac{x}{g}=\frac{\lambda x}{g}$ for all
$\lambda\in k$, $x\in H$ and $g\in G$.
\end{theorem}

\begin{proof}
In order to show~(S1), let $g\in G$ and $x\in H$. Then $g\in G$ and
$\mathcal{I}_g^{-1}(x)\in H$ satisfy $g\mathcal{I}_g^{-1}(x)=xg$. In order to
verify~(S2), let $g\in G$ and $x\in H$ such that $gx=0$. Then
\begin{equation}
0=\mathcal{I}_g^{-1}(gx)=\mathcal{I}_g^{-1}(\mathcal{I}_g(x)g)
=\mathcal{I}_g^{-1}(\mathcal{I}_g(x))\mathcal{I}_g^{-1}(g)
=x\mathcal{I}_g^{-1}(g)
\end{equation}
in which $\mathcal{I}_g^{-1}(g)\in G$. It is then straightforward to verify
that the operations are of the form claimed in the proposition and that
$H[G^{-1}]$ is indeed a vector space over $k$.
\end{proof}

The following Lemma establishes the analogue of Remark~\ref{rem_regular} in
our situation of an almost central set of denominators.

\begin{lemma}
\label{la_simplify}
Let $H$ be a WBA with an almost central multiplicative monoid $G\subseteq H$
via $\mathcal{I}\colon G\to\Aut(H)$. Then $(x,g)\sim (y,h)$ in $H[G^{-1}]\cong
(H\times G)/\sim$ (Theorem~\ref{thm_ringfractions}) if and only if there exist
$c,d\in G$ such that $xc=yd$ and $gc=hd$.
\end{lemma}

\begin{proof}
This condition is obviously sufficient. In order to see that it is necessary,
let $(x,g)\sim(y,h)$, \ie\ there exist $c,d\in H$ such that $xc=yd$ and
$gc=hd\in G$.

Then $g\mathcal{I}^{-1}_g(gc)=gcg$, \ie\ $g(\mathcal{I}^{-1}_g(gc)-cg)=0$, and
so by~(S2), there exists some $\ell\in G$ such that
$(\mathcal{I}^{-1}_g(gc)-cg)\ell=0$, \ie\
$cg\ell=\mathcal{I}^{-1}_g(gc)\ell$. Here, $gc,\ell\in G$, and so $cg\ell\in
G$, too. Similarly, $h\mathcal{I}^{-1}_h(hd)=hdh$, and there exists some $k\in G$ with
$dhk=\mathcal{I}^{-1}_h(hd)k\in G$.

We finally set $\tilde c=cg\ell\mathcal{I}^{-1}_{g\ell}(hk)\in G$ and $\tilde
d=dhkg\ell\in G$. They satisfy
\begin{eqnarray}
x\tilde c &=& xcg\ell\mathcal{I}^{-1}_{g\ell}(hk)
= ydg\ell\mathcal{I}^{-1}_{g\ell}(hk)
= ydhkg\ell = y\tilde d,\\
g\tilde c &=& gcg\ell\mathcal{I}^{-1}_{g\ell}(hk)
= hdg\ell\mathcal{I}^{-1}_{g\ell}(hk)
= hdhkg\ell = h\tilde d.
\end{eqnarray}
This proves the claim for $\tilde c$ and $\tilde d$.
\end{proof}

\begin{theorem}
\label{thm_localizewba2}
Let $H$ be a WBA with an almost central multiplicative monoid $G\subseteq H$
of group-like elements via $\mathcal{I}\colon G\to\Aut(H)$. Then the right
ring of fractions $H[G^{-1}]$ of Theorem~\ref{thm_localizewba1} forms a WBA
with the operations
\begin{eqnarray}
\Delta\biggl(\frac{x}{g}\biggr) &=& \frac{x^\prime}{g}\otimes\frac{x^\pprime}{g},\\
\epsilon\biggl(\frac{x}{g}\biggr) &=& \epsilon(x),
\end{eqnarray}
for all $x\in H$ and $g\in G$. Furthermore, the ring homomorphism $\phi\colon
H\to H[G^{-1}]$ is a homomorphism of WBAs.
\end{theorem}

\begin{proof}
In order to see that $\Delta$ is well defined, let $\frac{x}{g}=\frac{y}{h}$,
\ie\ $xc=yd$ and $gc=hd$ for some $c,d\in G$, using
Lemma~\ref{la_simplify}. Note that if $c\in H$ is group-like, then $x^\prime
c\otimes x^\pprime c=\Delta(xc)$ for all $x\in H$. Therefore,
\begin{eqnarray}
\Delta\biggl(\frac{x}{g}\biggr)&=&\frac{x^\prime}{g}\otimes\frac{x^\pprime}{g}
=\frac{x^\prime c}{gc}\otimes\frac{x^\pprime c}{gc}
=\biggl(\frac{x^\prime c}{1}\otimes\frac{x^\pprime c}{1}\biggr)
\biggl(\frac{1}{gc}\otimes\frac{1}{gc}\biggr)\nn\\
&=&\biggl(\frac{y^\prime d}{1}\otimes\frac{y^\pprime d}{1}\biggr)
\biggl(\frac{1}{hd}\otimes\frac{1}{hd}\biggr)
=\frac{y^\prime d}{hd}\otimes\frac{y^\pprime d}{hd}
=\frac{y^\prime}{h}\otimes\frac{y^\pprime}{h}=\Delta\biggl(\frac{y}{d}\biggr).
\end{eqnarray}
In order to see that $\epsilon$ is well defined, let $\frac{x}{g}=\frac{y}{h}$
as above. Note that if $c\in H$ is group-like, then $\epsilon(xc)=\epsilon(x)$
for all $x\in H$. Therefore,
\begin{equation}
\epsilon\biggl(\frac{x}{g}\biggr)
=\epsilon(x)=\epsilon(xc)=\epsilon(yd)=\epsilon(y)=\epsilon\biggl(\frac{y}{d}\biggr).
\end{equation}
It is then straightforward to show that both $\Delta$ and $\epsilon$ are
$k$-linear and that they satisfy the axioms of a WBA with multiplication and
unit given as in Theorem~\ref{thm_localizewba1}.

In order to see that $\phi\colon H\to H[G^{-1}]$, $x\mapsto\frac{x}{1}$, is a
homomorphism of WBAs, we confirm that
$(\phi\otimes\phi)(\Delta(x))=\phi(x^\prime)\otimes\phi(x^\pprime)
=\frac{x^\prime}{1}\otimes\frac{x^\pprime}{1}=\Delta(\frac{x}{1})=\Delta(\phi(x))$
for all $x\in H$ and that
$\epsilon(\phi(x))=\epsilon(\frac{x}{1})=\epsilon(x)$ for all $x\in H$.
\end{proof}

\begin{corollary}
\label{rem_bialgebra}
If $H$ in Theorem~\ref{thm_localizewba2} is a bialgebra, then $H[G^{-1}]$ is a
bialgebra as well.
\end{corollary}

\begin{proof}
If $H$ is a bialgebra, then $\epsilon_t(x)=1\epsilon(x)$ for all $x\in H$. In
this case, we find for all $\frac{x}{g}\in H[G^{-1}]$ that
\begin{eqnarray}
\epsilon_t\biggl(\frac{x}{g}\biggr)
&=& \epsilon\biggl(\frac{1^\prime}{1}\cdot\frac{x}{g}\biggr)\frac{1^\pprime}{1}
= \epsilon\biggl(\frac{1^\prime x}{g}\biggr)\frac{1^\pprime}{1}
= \epsilon(1^\prime x)\frac{1^\pprime}{1}
= \frac{\epsilon(1^\prime x)1^\pprime}{1}\nn\\
&=& \frac{\epsilon_t(x)}{1}
= \frac{1\epsilon(x)}{1}
= \frac{1}{1}\epsilon\biggl(\frac{x}{g}\biggr).
\end{eqnarray}
This in turn implies that $H[G^{-1}]$ is a bialgebra.
\end{proof}

\begin{proposition}
\label{prop_localizewba3}
Let $H$ be a WBA with an almost central multiplicative monoid $G\subseteq H$
of group-like elements via $\mathcal{I}\colon G\to\Aut(H)$. Then the right WBA
of fractions $H[G^{-1}]$ of Theorem~\ref{thm_localizewba2} satisfies the
following universal property: For each WBA $L$ and each homomorphism of WBAs
$\psi\colon H\to L$ such that $\psi(g)$ is invertible in $L$ for each $g\in
G$, there exists a unique homomorphism of WBAs $\sigma\colon H[G^{-1}]\to L$
such that the following diagram commutes:
\begin{equation}
\begin{aligned}
\xymatrix{
H\ar[rr]^{\phi}\ar[ddrr]_{\psi}&&H[G^{-1}]\ar[dd]^{\sigma}\\
\\
&&L.
}
\end{aligned}
\end{equation}
\end{proposition}

\begin{proof}
The underlying ring of $H[G^{-1}]$ and the unique map $\sigma$ are determined
by Proposition~\ref{prop_universal}. By Theorem~\ref{thm_localizewba2},
$H[G^{-1}]$ forms a WBA and $\phi$ a homomorphism of WBAs. It remains to
verify that $\sigma\colon H[G^{-1}]\to L$,
$\frac{x}{g}\to\psi(x){\psi(g)}^{-1}$ is a homomorphism of WBAs. For all
$x\in H$ and $g\in G$,
\begin{eqnarray}
(\sigma\otimes\sigma)\Delta\biggl(\frac{x}{g}\biggr)
&=& (\sigma\otimes\sigma)\biggl(\frac{x^\prime}{g}\otimes\frac{x^\pprime}{g}\biggr)
= \psi(x^\prime){\psi(g)}^{-1}\otimes\psi(x^\pprime){\psi(g)}^{-1}
= y^\prime\tilde g\otimes y^\pprime\tilde g\nn\\
&=& \Delta(y\tilde g)
= \Delta(\psi(x){\psi(g)}^{-1})
= \Delta(\sigma\biggl(\frac{x}{g}\biggr)).
\end{eqnarray}
Here, we have written $y=\psi(x)\in L$ and used the fact that $\tilde
g={\psi(g)}^{-1}\in L$ is group-like as the multiplicative inverse of the
group-like $\psi(g)\in L$. Finally,
\begin{eqnarray}
\epsilon(\sigma\biggl(\frac{x}{g}\biggr))
&=&\epsilon(\psi(x){\psi(g)}^{-1})
=\epsilon(y\tilde g)
=\epsilon(y)\nn\\
&=&\epsilon(\psi(x))=\epsilon(x)=\epsilon\biggl(\frac{x}{g}\biggr),
\end{eqnarray}
because $\tilde g\in L$ is group-like and $\psi$ a homomorphism of WBAs.
\end{proof}

We therefore call $H[G^{-1}]$ the \emph{right WBA of fractions} of $H$ with
\emph{denominator set} $G$. Let us finally add a WBA structure to
Example~\ref{ex_localizecentral}.

\begin{example}
\label{ex_localizecentral2}
Let $H$ be a WBA and $G=\{1,g,g^2,\ldots\}$ be the commutative monoid
generated by a central group-like element $g\in Z(H)$ of infinite order. In this
case, the WBA of fractions can be expressed in terms of polynomials with
coefficients in $H$ as follows,
\begin{equation}
H[X]/(gX-1)\cong H[G^{-1}].
\end{equation}
This is a special case of Example~\ref{ex_localizecentral}. The isomorphism of
the underlying rings, $\Phi\colon H[X]/(gX-1)\to H[G^{-1}]$ with
$\Phi(a)=\frac{a}{1}$ for all $a\in H$ and $\Phi(X)=\frac{1}{g}$  whose
inverse is given by $\Phi^{-1}(\frac{a}{x})=aX$ for all $a\in H$, can be used
in order to deduce the coalgebra structure on the polynomial algebra
$H[X]/(gX-1)$. We find on monomials:
\begin{eqnarray}
\Delta(aX^n) &=& a^\prime X^n\otimes a^\pprime X^n,\nn\\
\epsilon(aX^n) &=& \epsilon(a),
\end{eqnarray}
for all $a\in H$ and $n\in\{0,1,2,\ldots\}$. We just need to keep in mind that
$H$ is in general only a WBA and that therefore $\Delta$ and $\epsilon$ are not
necessarily homomorphisms of algebras.
\end{example}

\begin{remark}
\begin{myenumerate}
\item
If $H_1$ and $H_2$ are WBAs and $g_1\in H_1$ and $g_2\in H_2$ are group-like,
then $g_1\otimes g_2$ is group-like in the WBA $H_1\otimes H_2$.
\item
If $H_1$ and $H_2$ are coquasi-triangular WBAs with universal $r$-forms
$r_1\colon H_1\otimes H_1\to k$ and $r_2\colon H_2\otimes H_2\to k$, then
$H_1\otimes H_2$ forms a coquasi-triangular WBA with universal $r$-form
\begin{equation}
r\colon (H_1\otimes H_2)\otimes(H_1\otimes H_2)\to k,\quad
(x_1\otimes x_2)\otimes (y_1\otimes y_2)\mapsto r_1(x_1\otimes y_1)r_2(x_2\otimes y_2).
\end{equation}
\item
Let $H_1$ be a WBA with an almost central monoid $G_1\subseteq H_1$ of
group-like elements and $H_2$ be a WBA with an almost central monoid
$G_2\subseteq H_2$ of group-like elements. Then $G=\{g_1\otimes g_2|\, g_1\in
G_1, g_2\in G_2\}\subseteq H=H_1\otimes H_2$ is a monoid of group-like elements
and the associated right WBA of fractions is isomorphic as a WBA to
\begin{equation}
H[G^{-1}] \cong H_1[G_1^{-1}]\otimes H_2[G_2^{-1}]
\end{equation}
as can be seen from the universal property.
\end{myenumerate}
\end{remark}

\subsection{Relationship with the Hopf Envelope}

The following definition generalizes the notion of the \emph{Hopf closure} or
\emph{Manin envelope}~\cite{Ta71,Ma88} to WBAs.

\begin{definition}
Let $H$ be a WBA. A \emph{Hopf envelope} of $H$ is a WHA $\bar H$ with a
homomorphism of WBAs $\imath\colon H\to\bar H$ that satisfies the following
universal property. For every WHA $L$ and every homomorphism of WBAs $\psi\colon
H\to L$, there exists a unique homomorphism of WHAs $\sigma\colon\bar H\to L$ such
that the following diagram commutes:
\begin{equation}
\begin{aligned}
\xymatrix{
H\ar[rr]^{\imath}\ar[ddrr]_{\psi}&&\bar H\ar[dd]^{\sigma}\\
\\
&&L.
}
\end{aligned}
\end{equation}
\end{definition}

As a consequence, the Hopf envelope $\bar H$ is unique up to unique
isomorphism of WHAs as soon as it exists.

\begin{corollary}
\label{cor_envelope}
Let $H$ be a WBA such that the multiplicative monoid $G\subseteq H$ of all
group-like elements is almost central via $\mathcal{I}\colon
G\to\Aut(H)$. Then the Hopf envelope $\bar H$ factors through the right ring
of fractions $H[G^{-1}]$, \ie\ there exists a unique homomorphism of WBAs
$\rho\colon H[G^{-1}]\to\bar H$ such that
\begin{equation}
\begin{aligned}
\xymatrix{
H\ar[rr]^{\phi}\ar[ddrr]_{\imath}&&H[G^{-1}]\ar[dd]^{\rho}\\
\\
&&\bar H.
}
\end{aligned}
\end{equation}
commutes.
\end{corollary}

\begin{proof}
Since $\imath$ is a morphism of WBAs, $\imath(g)\in\bar H$ is group-like as
soon as $g\in H$ is. Since $\bar H$ is a WHA, $\imath(g)\in\bar H$ has got a
multiplicative inverse for each $g\in G\subseteq H$. By
Proposition~\ref{prop_localizewba3}, there exists a unique homomorphism of WBAs
$\rho$ such that the diagram commutes.
\end{proof}

\begin{remark}
\label{rem_hopfenvelope}
As a consequence of the universal property of the Hopf envelope, we see that
if $H[G^{-1}]$ happens to be a WHA, then it is already isomorphic to the Hopf
envelope, \ie\ $H[G^{-1}]\cong\bar H$.
\end{remark}

The following example shows that the Hopf envelope need not agree with the
localization at all group-like elements.

\begin{example}[Example~3 of~\cite{Ra80}]
The two-dimensional real vector space $V$ whose basis we denote by $\{1,i\}$
forms a coalgebra with $\Delta(1)=1\otimes 1-i\otimes i$,
$\Delta(i)=1\otimes i+i\otimes 1$, $\epsilon(1)=1$ and $\epsilon(i)=0$. This
is the coalgebra dual to $\C$ viewed as a $2$-dimensional algebra over
$\R$. The tensor algebra $H=T(V)=\R\oplus V\oplus(V\otimes V)\oplus\cdots$
then forms a bialgebra. It does not admit any antipode, and $1\in H$ is its
only group-like element. It is obviously invertible.

The right algebra of fractions whose denominator set is the monoid of all
group-like elements $G=\{1\}$ therefore agrees with $H$, \ie\ $H[G^{-1}]\cong
H$. Since we know that $H[G^{-1}]$ forms a bialgebra, this is the WBA of
fractions of the WBA $H$ with respect to $G$. However, $H[G^{-1}]$ does not
agree with the Hopf envelope $\bar H$ because it does not admit any antipode.
\end{example}

\section{Examples}
\label{sect_example}

In this section, we give a number of examples of WBAs of fractions. The first
subsection contains technical examples of bialgebras of fractions in order to
demonstrate that the assumptions of our main Theorem~\ref{thm_localizewba2}
are not superfluous. In the second subsection, we finally introduce a
genuinely new example of a WBA of fractions.

\subsection{Bialgebras}

Our first example is the standard non-commutative generalization of the
algebra of functions on complex $2\times 2$-matrices whose localization at the
(monoid generated by the) determinant yields the algebraic group $GL(2)$, see,
for example~\cite{Ka95}.

\begin{example}
\label{ex_mq2}
Let $q\in\C$ such that $q^2\neq-1$. By $M_q(2)$ we denote the quotient of the
free algebra $\C\{a,b,c,d\}$ by the two-sided ideal generated by the relations
\begin{gather}
ba=qab,\qquad db=qbd,\nn\\
ca=qac,\qquad dc=qcd,\nn\\
bc=cb,\qquad ad-da=(q^{-1}-q)bc.
\end{gather}
If we set $q=1$, the algebra $M_q(2)$ specializes to $M(2)=\C[a,b,c,d]$, the
commutative algebra of complex valued functions on the set of all $2\times
2$-matrices with complex coefficients.

The element $\det_q=da-qbc\in M_q(2)$ is called the \emph{quantum determinant} and can
be shown to be central. Using Example~\ref{ex_localizecentral}, we obtain the
localization
\begin{equation}
GL_q(2)=M_q(2)[X]/(\det_qX-1).
\end{equation}
Note that by setting $q=1$, the quantum determinant specializes to the usual
determinant of $2\times 2$-matrices, and the algebra $GL_q(2)$ becomes
commutative and specializes to $GL(2)$.

Since $M(2)$ consists of functions on a matrix algebra, it carries additional
operations that reflect this structure of a matrix algebra: $M(2)$ forms a
counital coassociative coalgebra with comultiplication and counit. The
analogous structure is present\footnote{$M_q(2)$ is a matrix coalgebra, \ie\
it is split cosimple.} in $M_q(2)$:
\begin{alignat}{2}
\Delta(a) &= a\otimes a+b\otimes c,&\qquad \Delta(b) &=a\otimes b+b\otimes c,\nn\\
\Delta(c) &= c\otimes a+d\otimes c,&\qquad \Delta(d) &=c\otimes b+d\otimes d,\\
\epsilon(a) &= 1,&\qquad \epsilon(b) &=0,\nn\\
\epsilon(c) &= 0,&\qquad \epsilon(d) &=1.
\end{alignat}
We can verify that $M_q(2)$ is a bialgebra and that
$\Delta(\det_q)=\det_q\otimes\det_q$ and $\epsilon(\det_q)=1$, \ie\ that
$\det_q$ is group-like. Finally, the bialgebra $M_q(2)$ is coquasi-triangular
with universal $r$-form $r\colon M_q(2)\otimes M_q(2)\to\C$ given by
\begin{gather}
r(a\otimes a)=q,\qquad r(a\otimes d)=1,\nn\\
r(d\otimes a)=1,\qquad r(d\otimes d)=q,\nn\\
r(c\otimes b)=q-q^{-1},
\end{gather}
and zero on all other tensor products of generators (keep in mind that
$r(1\otimes x)=1=r(x\otimes 1)$ for all $x\in H$ in any bialgebra as a
consequence of~\eqref{eq_propr1}). Its convolution inverse $\overline r\colon
M_q(2)\otimes M_q(2)\to\C$ is given by the same values on generators except
that $q$ and $q^{-1}$ have to be interchanged. Since $\det_q$ is central, the
monoid generated by $\det_q$ is trivially almost central with
$\mathcal{I}_{\det_q}=\id_{M_q(2)}$.

Theorem~\ref{thm_localizewba2} and Proposition~\ref{prop_localizewba3} confirm
that $GL_q(2)$ forms a WBA that satisfies the universal property of the right
WBA of fractions of $M_q(2)$ with denominators generated by $\det_q$. Since
$GL_q(2)$ is a WHA with antipode~\cite{Ka95}
\begin{alignat}{2}
S(a)&=dX,&\qquad S(b)&=-qbX,\nn\\
S(c)&=-q^{-1}cX,&\qquad S(d)&=aX,\nn\\
S(X)&=\det_q,
\end{alignat}
we know that $GL_q(2)=\bar{M_q(2)}$ coincides with the Hopf envelope of
$M_q(2)$ as well (Remark~\ref{rem_hopfenvelope}).
\end{example}

Our second example demonstrates that our notion of an almost central monoid of
group-likes is not superfluous. The example is the coquasi-triangular
bialgebra underlying Sweedler's Hopf algebra, see, for example~\cite[Example
2.2.6]{Ma95a}.

\begin{example}
\label{ex_sweedler}
Let $k$ be a field of characteristic $\Char k\neq 2$, $0\neq\alpha\in k$, and
$W_{k,\alpha}$ be the quotient of the free algebra $k\{f,y\}$ by the two-sided
ideal generated by the relations $y^2=0$, $f^2=1$ and $yf+fy=0$. This algebra
forms a bialgebra with $\Delta(f)=f\otimes f$, $\Delta(y)=y\otimes 1+f\otimes
y$, $\epsilon(f)=1$ and $\epsilon(y)=0$, and it has a coquasi-triangular
structure whose universal $r$-form is given by $r(f\otimes f)=-1$ and
$r(y\otimes y)=\alpha$ and zero on all other tensor products of
generators. The convolution inverse $\bar r$ of $r$ agrees with $r$ in this
case.

In this example, $f\in W_{k,\alpha}$ is a group-like element that is not
central. Nevertheless, the monoid $G_W=\{1,f\}\subseteq W_{k,\alpha}$ of
group-like elements is almost central in $W_{k,\alpha}$ by
Proposition~\ref{prop_almostcentral}. In this case, $\mathcal{I}_f(f)=f$ and
$\mathcal{I}_f(y)=-y$.
\end{example}

Since $W_{k,\alpha}$ is already a Hopf algebra with antipode $S(f)=f$ and
$S(y)=y$ and therefore $W_{k,\alpha}[G^{-1}]\cong W_{k,\alpha}$ (note that
$f\in W_{k,\alpha}$ is regular), the localization of this example alone is not
interesting. But we can tensor it with any other bialgebra with a monoid of
central group-likes, and obtain a tensor product not all of whose group-likes
are central. We use this idea in our third example:

\begin{example}
Let $0\neq\alpha\in\C$, $q\in\C$ such that $q^2\neq 1$, let $W_{\C,\alpha}$
denote Sweedler's bialgebra as in Example~\ref{ex_sweedler}, and $M_q(2)$ be
as in Example~\ref{ex_mq2}. Then $H=W_{\C,\alpha}\otimes M_q(2)$ forms a
coquasi-triangular bialgebra that does not admit any antipode. The monoid
$G\subseteq H$ generated by $1\otimes \det_q$ and $f\otimes 1$ consists of
group-like elements and is almost central by
Proposition~\ref{prop_almostcentral}, but the group-like element $f\otimes
1$ is not central. Theorem~\ref{thm_localizewba2} allows us to construct the
right bialgebra of fractions $H[G^{-1}]$. It is not difficult to see that in
this case, $H[G^{-1}]\cong W_{\C,\alpha}\otimes GL_q(2)$ as bialgebras.
\end{example}

The fourth example demonstrates that Lemma~\ref{la_simplify} is not
superfluous. Recall from Remark~\ref{rem_regular} that as soon as our ring
satisfies the right Ore condition and the denominator set consists of regular
elements only, the conclusion of Lemma~\ref{la_simplify} can already be
established using a different technique~\cite[Lemma A.5.3]{Be99}. Our example
contains a non-regular group-like element and shows that in this case,
Lemma~\ref{la_simplify} is not a trivial consequence of
Remark~\ref{rem_regular}.

\begin{example}
\label{ex_nonreg}
Let $M_4$ denote the multiplicative monoid of the factor ring $\Z/4\Z=\{\bar
0,\bar 1,\bar 2,\bar 3\}$. We write its operations for equivalence classes, \eg\
$\bar 2\cdot\bar 2=\bar 0$ (because $2\cdot 2\equiv 0$ mod $4$). We define its
complex monoid algebra $H_4=\C[M_4]$ in analogy to the notion of a group algebra,
\ie\ $H_4$ is the free complex vector space on the set $M_4$ with algebra unit
$\bar 1$, commutative multiplication induced from $M_4$, $\Delta(x)=x\otimes
x$ and $\epsilon(x)=1$ for all $x\in M_4$.

Note that $G=\{\bar 0,\bar 1\}\subseteq H_4$ forms a monoid of group-like
elements. Since $H_4$ is commutative as an algebra, $G$ is trivially almost
central. However, $\bar 0$ is not regular in $H_4$ because $\bar 0\cdot (\bar
1-\bar 0)=\bar 0\cdot\bar 1-\bar 0\cdot\bar 0=\bar 0-\bar 0=0$.

Our Theorem~\ref{thm_localizewba2} sill allows us to construct the right
bialgebra of fractions $H_4[G^{-1}]$. Using Lemma~\ref{la_simplify}, we can
compute the quotient $H_4[G^{-1}]=(H_4\times G)/\sim$ that appears in
Theorem~\ref{thm_ringfractions}. Writing
\begin{equation}
x=\sum_{i=0}^3x_i\bar i\in H_4\qquad\mbox{and}\qquad
y=\sum_{i=0}^3y_i\bar i\in H_4
\end{equation}
with $x_i,y_i\in\C$, we find that $(x,g)\sim(y,h)$ if and only if
\begin{equation}
\sum_{i=0}^3x_i=\sum_{i=0}^3y_i
\end{equation}
(irrespectively of $g,h\in G$). As a complex vector space, therefore
$H_4[G^{-1}]\cong\C$. Its bialgebra structure is obvious.
\end{example}

Although this bialgebra of fractions $H_4[G^{-1}]\cong\C$ is trivial,
we can use tensor products with $H_4$ in order to obtain non-trivial examples
of bialgebras of fractions with non-regular denominators as follows.

\begin{example}
Let $q\in\C$, $q^2\neq -1$, $M_q(2)$ be as in Example~\ref{ex_mq2} and $H_4$
as in Example~\ref{ex_nonreg}. Then $H=H_4\otimes M_q(2)$ forms a
coquasi-triangular bialgebra. The monoid $G\subseteq H$ generated by $\bar
0\otimes 1$, $\bar 1\otimes 1$ and $1\otimes\det_q$ consists of central group-like
elements. The element $\bar 0\otimes 1\in G$, for example, is not regular in
$H$. Nevertheless, Theorem~\ref{thm_localizewba2} allows us to construct a
right bialgebra of fractions. It turns out that in this case, $H[G^{-1}]\cong
M_q(2)$ as bialgebras.
\end{example}

\subsection{Weak Bialgebras}

The modular categories associated with $U_q(\mathfrak{sl}_2)$ at suitable
roots of unity arise as the categories of finite-dimensional comodules of
certain WHAs~\cite{Pf09a}. Let us call them $\hat{SL}_q(2)$, depending on the
root of unity $q$. In~\cite{Pf11}, we have constructed a WBA $H$ with a
central group-like element $\det_q$ in such a way that $\hat{SL}_q(2)$ agrees
with the quotient of $H$ modulo the relation $\det_q=1$. Let us call this WBA
$H=\hat{M}_q(2)$ in analogy with Example~\ref{ex_mq2}.

In the present section, we construct the WBA of fractions of $\hat{M}_q(2)$
relative to $\det_q$, \ie\ the WBA that deserves the name $\hat{GL}_q(2)$ in
this picture. This WBA is new and is presented here for the first time. Let us
begin by describing $\hat{M}_q(2)$ following~\cite{Pf11}.

\begin{example}
Let $r\in\{3,4,5,\ldots\}$ and $k=\Q(\epsilon)$ denote the cyclotomic
field associated with a primitive $4r$-th root of unity $\epsilon$. Let
$\sym{G}$ be the following finite directed graph
\begin{equation}
\sym{G} =
\begin{xy}
( 0, 0)*\dir{*};
( 0,-5)*{0};
( 1, 1);(14, 1)**\crv{( 5, 2)&(10, 2)} ?>*\dir{>};
( 1,-1);(14,-1)**\crv{( 5,-2)&(10,-2)} ?<*\dir{<};
(15, 0)*\dir{*};
(15,-5)*{1};
(16, 1);(29, 1)**\crv{(20, 2)&(25, 2)} ?>*\dir{>};
(16,-1);(29,-1)**\crv{(20,-2)&(25,-2)} ?<*\dir{<};
(30, 0)*\dir{*};
(30,-5)*{2};
(31, 1);(35, 1.5)**\crv{(33, 1.5)};
(31,-1);(35,-1.5)**\crv{(33,-1.5)} ?<*\dir{<};
(40, 0)*{\ldots};
(49, 1);(45, 1.5)**\crv{(47, 1.5)} ?<*\dir{<};
(49,-1);(45,-1.5)**\crv{(47,-1.5)};
(50, 0)*\dir{*};
(50,-5)*{r-3};
(51, 1);(64, 1)**\crv{(55, 2)&(60, 2)} ?>*\dir{>};
(51,-1);(64,-1)**\crv{(55,-2)&(60,-2)} ?<*\dir{<};
(65, 0)*\dir{*};
(65,-5)*{r-2};
\end{xy}
\end{equation}
We use the following notation and terminology. The set of vertices of
$\sym{G}$ is denoted by $\sym{G}^0$ and the set of edges by
$\sym{G}^1\subseteq\sym{G}^0\times\sym{G}^0$. Every edge
$p=(v_0,v_1)\in\sym{G}^1$ has a source and a target vertex, denoted by
$\sigma(p)=v_1\in\sym{G}^0$ and $\tau(p)=v_0\in\sym{G}^0$, respectively. We
also set $\sigma(v)=v=\tau(v)$ for all $v\in\sym{G}^0$. By
\begin{equation}
\sym{G}^m=\{\,(p_1,\ldots,p_m)\in{(\sym{G}^1)}^m\mid\quad
\sigma(p_j)=\tau(p_{j+1})\quad\mbox{for all}\quad 1\leq j\leq m-1\,\},
\end{equation}
we denote the set of paths of length $m$ in $\sym{G}$, $m\in\N$ (the set of
natural numbers without zero). Finally, for vertices $v,w\in\sym{G}^0$, the set
\begin{equation}
\sym{G}^m_{wv}=\{\,p\in\sym{G}^m\mid\quad\sigma(p)=v,\quad\tau(p)=w\,\}
\end{equation}
contains all paths of length $m\in\N_0=\N\cup\{0\}$ from $v$ to $w$.

We write $pq\in\sym{G}^{\ell+m}$ for the concatenation of two paths
$p\in\sym{G}^\ell$ and $q\in\sym{G}^m$ provided that $\sigma(p)=\tau(q)$. The
free $k$-vector space on the set $\sym{G}^m$ is denoted by $k\sym{G}^m$,
$m\in\N_0$.

The vector space
\begin{equation}
H[\sym{G}] = \coprod_{m\in\N_0}{(k\sym{G}^m)}^\ast\otimes k\sym{G}^m
\end{equation}
forms a WBA with the operations
\begin{eqnarray}
\eta(1)
&=& \sum_{j,\ell\in\sym{G}^0}{[j|\ell]}_0,\\
\mu({[p|q]}_m\otimes {[r|s]}_\ell)
&=& \delta_{\sigma(p),\tau(r)}\delta_{\sigma(q),\tau(s)}{[pr|qs]}_{m+\ell},\\
\Delta({[p|q]}_m)
&=& \sum_{r\in\sym{G}^m}{[p|r]}_m\otimes {[r|q]}_m,\\
\epsilon({[p|q]}_m)
&=& \delta_{pq},
\end{eqnarray}
for all $p,q\in\sym{G}^m$, $r,s\in\sym{G}^\ell$, $m,\ell\in\N_0$.  Here,
$\coprod$ is the coproduct in the category $\Vect_k$ of $k$-vector spaces, and
we have denoted basis vectors of the homogeneous components
${H[\sym{G}]}_m={(k\sym{G}^m)}^\ast\otimes k\sym{G}^m$ of $H[\sym{G}]$ by
${[p|q]}_m=p\otimes q\in {(k\sym{G}^m)}^\ast\otimes k\sym{G}^m$. As usual, we
write $\delta_{pq}=1$ if $p=q$ and $\delta_{pq}=0$ if $p\neq q$ for all
$p,q\in\sym{G}^m$, $m\in\N_0$.

Since for any two vertices $j,\ell\in\sym{G}^0$, there is at most one edge
from $j$ to $\ell$, we specify a path $p\in\sym{G}^m$ of length $m\in\N_0$ by
the sequence of the $m+1$ vertices along $p$, \ie\ $p=(i_0,\ldots,i_m)\in
{(\sym{G}^0)}^{m+1}$. The source and target of this path are $\sigma(p)=i_m$
and $\tau(p)=i_0$.

Observe that as an algebra, $H[\sym{G}]\cong k(\sym{G}\times\sym{G})$ is the
path algebra of the quiver $\sym{G}\times\sym{G}$, and as such it is graded
with respect to path length $m\in\N_0$. As a coalgebra, it is a direct sum of
matrix coalgebras: one for each degree, \ie\ for each length of paths.

The WBA $\hat{M}_q(2)$ we are interested in, is the quotient
$\hat{M}_q(2)=H[\sym{G}]/I$ by the two-sided ideal $I$ generated by the
relations
\begin{eqnarray}
\label{eq_rtt}
& & \sum_{(i_0,i_1,i_2)\in\sym{G}^2}
{[(j_0,j_1,j_2)|(i_0,i_1,i_2)]}_2\,R_{(i_0,i_1,i_2);(\ell_0,\ell_1,\ell_2)}\nn\\
&=& \sum_{(i_0,i_1,i_2)\in\sym{G}^2}
R_{(j_0,j_1,j_2);(i_0,i_1,i_2)}\,{[(i_0,i_1,i_2)|(\ell_0,\ell_1,\ell_2)]}_2,
\end{eqnarray}
for all paths of length two $(j_0,j_1,j_2)\in\sym{G}^2$ and
$(\ell_0,\ell_1,\ell_2)\in\sym{G}^2$. Here, the coefficients
$R_{(i_0,i_1,i_2);(\ell_0;\ell_1;\ell_2)}\in k$ are given by
\begin{eqnarray}
R_{(j,j\pm1,j);(j,j\pm1,j)}
&=& \mp q^{-1/2}\,\frac{q^{\pm(j+1)}}{\q{j+1}},\\
R_{(j,j-1,j);(j,j+1,j)}
&=& q^{-1/2}\frac{\q{j}\q{j+2}}{{\q{j+1}}^2},\\
R_{(j,j+1,j);(j,j-1,j)}
&=& q^{-1/2},\\
R_{(j,j\pm1,j\pm2);(j,j\pm1,j\pm2)} &=& q^{-3/2},
\end{eqnarray}
and zero for all other indices. We denote by $\q{n}=(q^n-q^{-n})/(q-q^{-1})$,
$n\in\Z$, the usual quantum integers for $q=\epsilon^2$.

We remark that if $\sym{G}$ has only one vertex, then $H[\sym{G}]$ forms a
bialgebra and coincides with the tensor algebra on the vector space of
$|\sym{G}^1|\times|\sym{G}^1|$-matrices, \ie\ with the free algebra
$H[\sym{G}]\cong k\{\,t_{pq}\mid\quad p,q\in\sym{G}^1\,\}$. The quotient
modulo the relations~\eqref{eq_rtt} then reduces to the
Faddeev--Reshetikhin--Takhtadjan construction~\cite{ReTa90}.

The WBA $\hat{M}_q(2)$ contains a central group-like element
\begin{equation}
\label{eq_qdet}
\begin{split}
\det_q=\sum_{j,\ell=0}^{r-2}\alpha_j\alpha_\ell\,\Biggl(
&\frac{\q{\ell+1}}{\q{j+1}}\,{[(j,j+1,j)|(\ell,\ell+1,\ell)]}_2
+\frac{\q{\ell}}{\q{j}}\,{[(j,j-1,j)|(\ell,\ell-1,\ell)]}_2\\
-&\frac{\q{\ell+1}}{\q{j}}\,{[(j,j-1,j)|(\ell,\ell+1,\ell)]}_2
-\frac{\q{\ell}}{\q{j+1}}\,{[(j,j+1,j)|(\ell,\ell-1,\ell)]}_2
\Biggr),
\end{split}
\end{equation}
with $\alpha_0=\alpha_{r-2}=1$ and $\alpha_j=1/\sqrt{2}$ for all $1\leq j\leq
r-3$ (the square root is not required in $k$; we just use this notation in
order to render the expression~\eqref{eq_qdet} most
symmetric). In~\eqref{eq_qdet} it is understood that terms with a path
$(j,j\pm1,j)$ are omitted from the expression whenever $j\pm1<0$ or
$j\pm1>r-2$.
\end{example}

If we denote by $\hat{SL}_q(2)$ the WHA whose category of finite-dimensional
comodules is the modular category associated with $U_q(\mathfrak{sl}_2)$, then indeed
\begin{equation}
\hat{SL}_q(2)\cong \hat{M}_q(2)/(\det_q-1),
\end{equation}
as the notation suggests. The following example finally introduces the new WBA
that plays the role of $\hat{GL}_q(2)$ in this picture.

\begin{example}
\label{ex_fusion}
Let $\hat{M}_q(2)$ be as above and $G\subseteq\hat{M}_q(2)$ be the monoid
generated by $\det_q$. $G$ therefore consists of central group-like
elements and is generated by an element of infinite
order. Example~\ref{ex_localizecentral2} then provides the right WBA of
fractions as a quotient of a polynomial algebra:
\begin{equation}
\hat{GL}_q(2) = \hat{M}_q(2)[G^{-1}] \cong \hat{M}_q(2)[X]/(\det_q X-1).
\end{equation}
Note that neither $\hat{M}_q(2)$ nor $\hat{GL}_q(2)$ are bialgebras. Similarly
to $GL_q(2)$, $\hat{GL}_q(2)$ forms a WHA, and so it already agrees with the
weak Hopf envelope $\bar{\hat{M}_q(2)}$.
\end{example}

We can use this example in order to obtain examples of WBAs of fractions for
almost central, but not central denominators and of WBAs of fractions that use
non-regular denominators. The most general example that probes all key
assumptions of Theorem~\ref{thm_localizewba2}, is the following.

\begin{example}
\label{ex_glfusion}
Let $r\in\{3,4,5,\ldots\}$, $\epsilon\in\C$ be a primitive $4r$-th root of
unity, $q=\epsilon^2$ and $0\neq\alpha\in\C$. We denote by $W_{\C,\alpha}$ the
coquasi-triangular bialgebra underlying Sweedler's Hopf algebra of
Example~\ref{ex_sweedler}, by $H_4$ the commutative bialgebra of
Example~\ref{ex_nonreg}, and by $\hat{M}_q(2)$ the coquasi-triangular WBA of
Example~\ref{ex_fusion}. Then $H=W_{\C,\alpha}\otimes H_4\otimes\hat{M}_q(2)$
is a coquasi-triangular WBA that is neither commutative nor a bialgebra.

Let $G\subseteq H$ be the monoid generated by $f\otimes\bar 1\otimes 1$,
$y\otimes\bar 1\otimes 1$, $1\otimes\bar 0\otimes 1$ and $1\otimes\bar
1\otimes\det_q$. Then $G$ contains the element $f\otimes 1\otimes 1$ which is
not central and the element $1\otimes\bar 0\otimes 1$ which is not
regular. Theorem~\ref{thm_localizewba2} is now required in full generality and
allows us to construct the right WBA of fractions which turns out to be
isomorphic to
\begin{equation}
H[G^{-1}] \cong W_{\C,\alpha}\otimes\hat{GL}_q(2).
\end{equation}
as WBAs. Note that $H[G^{-1}]$ is neither commutative nor a bialgebra either.
\end{example}

\begin{remark}
In the broader context of a research programme, Example~\ref{ex_fusion}
can be seen as the main motivation for the present article. From our work on
fusion categories~\cite{Pf11}, we know that we can obtain a dimension graph
$\sym{G}$ from each fusion category $\sym{C}$ with a choice of a monoidal
generating object $M\in|\sym{C}|$. There is a `free' WBA $H[\sym{G}]$
associated with $\sym{G}$ and a quotient of $H[\sym{G}]$ modulo some
$RTT$-relations that plays the role of $\hat M_q(2)$. The WHA $H$ that
characterizes the fusion category $\sym{C}$, is finally a quotient of $\hat
M_q(2)$ modulo setting certain group-like elements to $1$.

We can now ask whether we can turn this argument upside down and approach the
question of classifying (or at least systematically constructing plenty of)
fusion categories as follows. Start with a finite directed graph
$\sym{G}$. Consider the WBA $H[\sym{G}]$ and study solutions to the weak
version of the Quantum Yang--Baxter equation which yields $RTT$-relations that
can be imposed. Any finite-dimensional quotient of the weak Hopf envelope of
the resulting quotient WBA then produces a multi-fusion category as its
category of finite-dimensional comodules.
\end{remark}

\subsubsection*{Acknowledgements}

We are grateful to Ralph on \textbf{math}\textsl{overflow} for directing us to
relevant literature and to Yorck Sommerh{\"a}user for correspondence.

\newenvironment{hpabstract}{%
\renewcommand{\baselinestretch}{0.2}
\begin{footnotesize}%
}{\end{footnotesize}}%
\newcommand{\hpeprint}[2]{%
\href{http://www.arxiv.org/abs/#1}{\texttt{arxiv:#1#2}}}%
\newcommand{\hpspires}[1]{%
\href{http://www.slac.stanford.edu/spires/find/hep/www?#1}{SPIRES Link}}%
\newcommand{\hpmathsci}[1]{%
\href{http://www.ams.org/mathscinet-getitem?mr=#1}{\texttt{MR #1}}}%
\newcommand{\hpdoi}[1]{%
\href{http://dx.doi.org/#1}{\texttt{DOI #1}}}%
\newcommand{\hpjournal}[2]{%
\href{http://dx.doi.org/#2}{\textsl{#1\/}}}%

\end{document}